\documentclass[11pt,a4paper]{article}
\usepackage{multibib}
\usepackage{ulem}
\usepackage{xcolor}
\usepackage{CJK}
\usepackage{float}
\usepackage[export]{adjustbox}
\usepackage{graphicx}
\graphicspath{ {./images/} }
\usepackage{subcaption}
\usepackage[utf8]{inputenc}
\usepackage[export]{adjustbox}
\usepackage{wrapfig}
\usepackage{authblk}
\usepackage{pgfplots}
\usepackage[top=27mm, bottom=27mm, left=22mm, right=22mm]{geometry}
\usepackage{amsthm,amsmath,amssymb}
\usepackage{mathrsfs}
\usepackage{multirow}
\usepackage{amsthm}
\usepackage{microtype}
\usepackage{hyperref}
\usepackage[capitalize]{cleveref}
\usepackage{xcolor}
\usepackage{verbatim}
\usepackage{indentfirst}
\usepackage{cite}

\usepackage{ulem}

\newtheorem{theorem}{Theorem}[section]
\newtheorem{lemma}[theorem]{Lemma}
\newtheorem{proposition}[theorem]{Proposition}

\newtheorem{definition}[theorem]{Definition}

\newtheorem{claim}[theorem]{Claim}

\newtheorem{question}[theorem]{Question}

\newcommand{\ma}{\mathcal}

\newcommand{\bi}{\binom}

\DeclareMathOperator{\cod}{cod}
\DeclareMathOperator{\supp}{supp}
\DeclareMathOperator{\wt}{wt}

\title{Approximate generalized Steiner systems and near-optimal constant weight codes}
\date{}

\author{Miao Liu\thanks{M. Liu is with the Research Center for Mathematics and Interdisciplinary Sciences, Shandong University, Qingdao 266237, P.R. China (e-mail: liumiao10300403@163.com).} and Chong Shangguan\thanks{C. Shangguan is with the Research Center for Mathematics and Interdisciplinary Sciences, Shandong University, Qingdao 266237, China, and the Frontiers Science Center for Nonlinear Expectations, Ministry of Education, Qingdao 266237, China (e-mail: theoreming@163.com).}
}

\begin{document}

\maketitle

\begin{abstract}
\noindent  Constant weight codes (CWCs) and constant composition codes (CCCs) are two important classes of codes that have been studied extensively in both combinatorics and coding theory for nearly sixty years. In this paper we show that for {\it all} fixed odd distances, there exist near-optimal CWCs and CCCs asymptotically achieving the classic Johnson-type upper bounds.

Let $A_q(n,w,d)$ denote the maximum size of $q$-ary CWCs of length $n$ with constant weight $w$ and minimum distance $d$. One of our main results shows that for {\it all} fixed $q,w$ and odd $d$, one has $\lim_{n\rightarrow\infty}\frac{A_q(n,d,w)}{\binom{n}{t}}=\frac{(q-1)^t}{\binom{w}{t}}$, where $t=\frac{2w-d+1}{2}$. This implies the existence of near-optimal generalized Steiner systems originally introduced by Etzion, and can be viewed as a counterpart of a celebrated result of R\"odl on the existence of near-optimal Steiner systems. Note that prior to our work, very little is known about $A_q(n,w,d)$ for $q\ge 3$. A similar result is proved for the maximum size of CCCs.

We provide different proofs for our two main results, based on two strengthenings of the well-known Frankl-R\"odl-Pippenger theorem on the existence of near-optimal matchings in hypergraphs: the first proof follows by Kahn's linear programming variation of the above theorem, and the second follows by the recent independent work of Delcour-Postle, and Glock-Joos-Kim-K\"uhn-Lichev on the existence of near-optimal matchings avoiding certain forbidden configurations.

We also present several intriguing open questions for future research.
\vspace{10pt}


\vspace{10pt}

\end{abstract}

%\tableofcontents

\section{Introduction}\label{sec:cwc-ccc}

\subsection{Constant weight codes and constant composition codes}

\noindent Constant weight codes (CWCs) and constant composition codes (CCCs) are two important classes of codes with lots of applications, see e.g. \cite{Cos-For07,Che-Pur12,Chu-Duk06,Chu-Duk04}. They have been studied extensively in both combinatorics and coding theory for nearly sixty years, see e.g. \cite{Wood02-structure,Agr-Zeg00,Bro-War90,Gra-Slo80,Etzion-97,Che-Zha19,Chu-Duk06,Din-Yin05,Luo-Che03,Ding08}. Determining the maximum sizes of these two codes are the fundamental problems in this research area. Thanks to the well-known equivalence between binary CWCs and packings (see below for details), a vast number of results exist for binary CWCs. However, for $q>3$, relative little is known about $q$-ary CWCs and CCCs. In this paper, we take a step towards this direction by exploiting the deep connection between these two codes and matchinges in certain appropriately defined hypergraphs. We show that for a vast set of parameters (all fixed odd distances), there exist near-optimal CWCs and CCCs asymptotically achieving the classic Johnson-type upper bounds for these codes, up to a lower order term.

To proceed let us begin with some standard notations in coding theory. For positive integers $q$ and $n$, let $\Sigma_q:=\{0,1,\dots,q-1\}$ and $[n]:=\{1,2,\dots,n\}$. Let $\Sigma_q^n$ be the set of all vectors of length $n$ over the alphabet $\Sigma_q$. A $q$-ary code of length $n$ is a subset of $\Sigma_q^n$. A vector $\mathbf{x}\in\Sigma_q^n$ is denoted by $\mathbf{x}=(x_1,\ldots,x_n)$. The {\it weight} of $\mathbf{x}$, denoted by $\wt(\mathbf{x})$, is the number of non-zero coordinates of $\mathbf{x}$. Let $w,w_1,\ldots,w_{q-1}$ be non-negative integers. The {\it composition} of $\mathbf{x}$ is the vector\footnote{We will use a bar over a lowercase letter, e.g., $\overline{w}=(w_1,\dots,w_{q-1})$, to denote a composition of some vector, which particularly implies that $w,w_1,\ldots,w_{q-1}$ are all non-negative integers with $w=\sum_{j=1}^{q-1}w_j$.} $\overline{w}=(w_1,\dots,w_{q-1})\in [n]^{q-1}$, where for each $j\in[q-1]$, $w_j=|\{i\in [n]:x_i=j\}|$ is the number of coordinates of $\mathbf{x}$ that equal to $j$. For two vectors $\mathbf{x},\mathbf{y}\in \Sigma_q^n$, the {\it Hamming distance} $d(\mathbf{x},\mathbf{y}):=|\{i\in [n]:x_i\neq y_i\}|$ is the number of coordinates for which $\mathbf{x}$ and $\mathbf{y}$ differ. A code $C\subseteq\Sigma_q^n$ has {\it minimum distance} $d$ if $d(\mathbf{x},\mathbf{y})\ge d$ for all distinct $\mathbf{x},\mathbf{y}\in C$.

A code $C\subseteq\Sigma_q^n$ is said to be a {\it constant weight code} with constant weight $w$ if for every $\mathbf{x}\in C$, $\wt(\mathbf{x})=w$; similarly, it is is said to be a {\it constant composition code}\footnote{Clearly, such a code also has constant weight $w=\sum_{j=1}^{q-1}w_j$.} with constant composition $\overline{w}=(w_1,\dots,w_{q-1})$ if every $\mathbf{x}\in C$ has composition $\overline{w}$. For convenience, a code $C\subseteq\Sigma_q^n$ with minimum distance $d$ and constant weight $w$ (resp. constant composition $\overline{w}$) is denoted by an $(n,d,w)_q$-CWC (resp. $(n,d,\overline{w})_q$-CCC). Let $A_q(n,d,w)$ (resp. $A_q(n,d,\overline{w})$) denote the maximum cardinality among all $(n,d,w)_q$-CWCs (resp. $(n,d,\overline{w})_q$-CCCs).

%There are several interesting parameter regimes in the study of these two functions, see e.g. \cite{Che-Lin10,Che-Zha19,Zha-Ge10,Zha-Ge12,Zha-Ge13,Che-Lin07,Che-Zha15,Che-Lin08}.

Determining the exact or asymptotic values of $A_q(n,d,w)$ and $A_q(n,d,\overline{w})$ are central problems in the study of CCCs and CWCs. We will focus on the asymptotic behaviour of the above two functions when $q,d,w,\overline{w}$ are all {\it fixed} and $n$ tends to {\it infinity}. It is worth mentioning that this set of parameters are of particular importance as they have deep connections with many well-studied objects in extremal and probabilistic combinatorics, as well as combinatorial design theory. In fact, the equivalence between packings and binary CWCs completely determines the asymptotic values of $A_q(n,d,w)$ for $q=2$. Moreover, our work is essentially based on some powerful developments in extremal and probabilistic combinatorics, which consider the sufficient conditions on which a given hypergraph has a near-perfect matching avoiding certain forbidden configurations. We will discuss in detail the connections between these two codes and various combinatorial objects below.
%in \cref{sec:related_structure}.

\subsection{Related combinatorial structures}\label{sec:related_structure}

\noindent We briefly mention some combinatorial structures that are closely related to CWCs or CCCs.

\noindent \paragraph{Packings, Steiner systems, and binary CWCs.} Let $\binom{[n]}{w}=\{A\subseteq [n]:|A|=w\}$ be the family of all $w$-subsets of $[n]$. A family of $w$-subsets $\mathcal{P}\subseteq\binom{[n]}{w}$ is called an {\it $(n,w,t)$-packing} if for every distinct $A,B\in\mathcal{P}$, $|A\cap B|<t$. The packing function $P(n,w,t)$ is the maximum possible cardinality of an $(n,w,t)$-packing. By a double counting argument it is not hard to check that $P(n,w,t)\le \binom{n}{t}/\binom{w}{t}$. An $(n,w,t)$-packing $\mathcal{P}$ is called {\it asymptotically good} if $|\mathcal{P}|\ge(1-o(1))\cdot\binom{n}{t}/\binom{w}{t}$, where $o(1)\rightarrow 0$ as $n\rightarrow\infty$; moreover, it is called an {\it $(n,w,t)$-Steiner system} if $|\mathcal{P}|=\binom{n}{t}/\binom{w}{t}$. In 1963, Erd\H{o}s and Hanani \cite{Erd-Han63} conjectured that for all fixed $2\le t<w$,
\begin{align*}
\lim_{n\rightarrow\infty}\frac{ P(n,w,t)}{\binom{n}{t}/\binom{w}{t}}=1,
\end{align*}
i.e., there always exist asymptotically good $(n,w,t)$-packings. In 1985, R\"{o}dl \cite{Rod85} introduced the celebrated nibble method and resolved the Erd\H{o}s-Hanani conjecture. Strengthening R\"{o}dl's result, in a breakthrough in 2014 Keevash \cite{Kee14} showed that one can always have $P(n,w,t)=\binom{n}{t}/\binom{w}{t}$, i.e., $(n,w,t)$-Steiner systems always exist, provided that $n$ is sufficiently large and satisfies certain necessary divisibility condition. Keevash's result resolved one of the oldest problems in combinatorics, originally asked by Steiner from 1853. The papers \cite{Rod85,Kee14} have been very influential, and motivated a lot of research in extremal and probabilistic combinatorics, combinatorial design theory, and related fields, see the surveys \cite{Glock21suervey,Kang23survey,Keevash18survey} and the references therein.

Packings and binary CWCs are in fact equivalent. To see this, for every $A\in\binom{[n]}{w}$ let $\sigma(A)=(a_1,\ldots,a_n)\in\{0,1\}^n$ denote the {\it characteristic vector} of $A$ which is the binary vector defined as for each $i\in[n]$, $a_i=1$ if and only if $i\in A$. Clearly, $\sigma$ defines an one-to-one correspondence between $\binom{[n]}{w}$ and binary vectors of length $n$ with weight $w$. It is straightforward to check that for $A,B\in\binom{[n]}{w}$,
\begin{align*}
    2|A\cap B|+d(\sigma(A),\sigma(B))=2w.
\end{align*}
Therefore, $\mathcal{P}\subseteq\binom{[n]}{w}$ is an $(n,w,t)$-packing if and only if $\sigma(\mathcal{P}):=\{\sigma(A):A\in\mathcal{P}\}\subseteq\{0,1\}^n$ is an $(n,2w-2t+2,w)_2$-CWC.

The equivalence between packings and binary CWCs therefore implies that $A_2(n,d,w)=P(n,w,\frac{2w-d+2}{2})$.
%\footnote{For $A_2(n,d,w)$ we only need to consider the case when $d$ is even since the Hamming distance between arbitrary two binary vectors with the same weight can only be even.}
Hence, the aforementioned results of R\"{o}dl and Keevash also gave fairly satisfactory resolution to the problem of determining $A_2(n,d,w)$ for all fixed $w,d$ as $n\rightarrow\infty$.

\paragraph{Generalized Steiner systems and $q$-ary CWCs.} The equivalence between packings and binary CWCs leads to a satisfactory solution to the problem of estimating $A_2(n,d,w)$ in the aforementioned parameter regime. It is therefore natural to ask what happens for $q\ge 3$. In particular, whether or not there exists a ``free'' lower bound for $A_q(n,d,w)$ when $q\ge 3$. Indeed, for $q\ge 3$ Etzion \cite{Etz18} defined {\it generalized Steiner systems} as $(n,d,w)$-CWCs with odd $d$ that achieve the upper bound in \cref{lem Ucwc} with equality. However, there was no general existential result that is in parallel with R\"{o}dl's asymptotic lower bound or Keevash's exact result. The interested reader is referred to a book of Etzion \cite{Etz22} for background and many open problems on generalized Steiner systems. \cref{thm:cwc} in fact implies the existence of near-optimal generalized Steiner systems, and can be viewed as a counterpart of R\"odl's result \cite{Rod85} in a more general setting.

Lastly, Chee, Gao, Kiah, Ling, Zhang and Zhang \cite{Che-Zha19} used decompositions of edge-colored digraphs as a new technique to obtain asymptotically exact results for $A_q(n,d,w)$ and $A_q(n,d,\overline{w})$. However they only considered $d\in\{2w-2,2w-3\}$.

\paragraph{Matchinges in hypergraphs.} R\"odl's result \cite{Rod85} can be reformulated as a special case of the more general results of Frankl and R\"odl \cite{Fra-Rod85} and Pippenger (see \cite{Pip-Spe89}) on the the existence of near-optimal perfect matchinges in hypergraphs. In the literature, their important results were often termed as the Frankl-R\"odl-Pippenger theorem, see e.g. \cite{Kah96}. Our proofs of Theorems \ref{thm:cwc} and \ref{thm:ccc} are essentially based on two powerful strengthenings of the Frankl-R\"odl-Pippenger theorem, one of which is Kahn's linear programming variation of the above theorem \cite{Kah96}, and the other is the very recent independent work of Delcour and Postle \cite{Del-Pos22}, and Glock, Joos, Kim, K\"uhn, and Lichev \cite{Glo-Lic23} on the existence of near-optimal matchings avoiding certain forbidden configurations. We will discuss the connection between codes and matchinges in hypergraphs, as well as the results of \cite{Kah96,Del-Pos22,Glo-Lic23} in \cref{sec:matchinges} below.

\subsection{Upper and lower bounds on $A_q(n,d,w)$ and $A_q(n,d,\overline{w})$}\label{subsec:bounds}

\noindent In the literature, there are a lot of papers devoting to the determination of the exact or asymptotic values of $A_q(n,d,w)$ and $A_q(n,d,\overline{w})$ for the above set of parameters. Below we will briefly review some previously known upper and lower bounds on $A_q(n,d,w)$ and $A_q(n,d,\overline{w})$, and then introduce our main results.

In the remaining part of this paper, we always assume that $q,d,w,\overline{w}$ are fixed and $n\rightarrow\infty$. The Johnson-type upper bounds for $A_q(n,d,w)$ and $A_q(n,d,\overline{w})$, which followed from a classic argument of Johnson \cite{Joh62}, are presented as follows.
%We assume without loss of generality that $w_1\ge \dots \ge w_{q-1}$.

\begin{lemma}[Johnson-type bounds for CWCs and CCCs, \cite{Ost-Sva02,Sva-Gal02}]\label{lem:recursive}
    \begin{align*}
        A_q(n,d,w)&\le \left\lfloor \frac{(q-1)n}{w}A_q(n-1,d,w-1) \right\rfloor,\\
         A_q(n,d,\overline{w})&\le \left\lfloor \frac{n}{w_1}A_q(n-1,d,(w_1-1,w_2,\dots,w_{q-1}))\right\rfloor.
    \end{align*}
\end{lemma}

Applying \cref{lem:recursive} recursively it is not hard to deduce the following upper bound for CWCs.

\begin{theorem}[Upper bound for CWCs, see e.g. \cite{Che-Lin07}]\label{lem Ucwc}
    Let $t=\lceil \frac{2w-d+1}{2} \rceil$. Then
    \begin{equation*}
        A_q(n,d,w)\le \left\{ \begin{array}{cl}
                   \frac{(q-1)^t \binom{n}{t}}{\binom{w}{t}},& \text{if $d$ is odd;} \\
                    \frac{(q-1)^{t-1} \binom{n}{t}}{\binom{w}{t}}, & \text{if $d$ is even}.
                \end{array}\right.
    \end{equation*}
\end{theorem}

Some more notations are needed before presenting the upper bound for CCCs. Given two positive integers $w,t$ and a vector $\overline{w}=(w_1,\dots,w_{q-1})$. We say that a vector $\overline{t}=(t_1,\dots,t_{q-1})$ is {\it $(\overline{w},t)$-admissible} if $\sum_{i=1}^{q-1} t_i=t$ and $t_i\le w_i$ for all $1\le i\le q-1$. Let
\begin{equation}\label{eq f}
     f(\overline{w},t)=\max\limits_{\overline{t} \text{ is }(\overline{w},t) \text{-admissible}}\binom{t}{t_1,\dots ,t_{q-1}}=\max\limits_{\overline{t} \text{ is }(\overline{w},t) \text{-admissible}}~\frac{t!}{t_1!\cdots t_{q-1}!}.
\end{equation}

%where the {\it multinomial coefficient} is defined by
%$$\binom{t}{t_1,\dots,t_{q-1}}=\binom{t}{t_1}\binom{t-t_1}{t_2}\cdots \binom{t-t_1-\cdots -t_{q-2}}{t_{q-1}}=\frac{t!}{t_1!\cdots t_{q-1}!}.$$

\begin{theorem}[Upper bound for CCCs, see Lemma 5 in \cite{Luo-Che03}] \label{lem Uccc} % \cite{Luo-Che03}
    Let $w_0=n-w$ and $t=\lceil \frac{2w-d+1}{2} \rceil$. Then we have
    \begin{equation*}
        A_q(n,d,\overline{w})\le \left\{ \begin{array}{cl}
                   \frac{\binom{n}{n-w,w_1,\dots, w_{q-1}}}{\binom{n-t}{w-t}f(\overline{w},w-t)},& \text{if $d$ is odd}; \\
                   \frac{\binom{n}{n-w,w_1,\dots, w_{q-1}}}{\binom{n-t}{w-t}f(\overline{w},w-t+1)}, & \text{if $d$ is even}.
                \end{array}\right.
    \end{equation*}
\end{theorem}

%Theorems \ref{lem Ucwc} and \ref{lem Uccc} are both easy consequences of \cref{lem:recursive}, see e.g. \cite{Che-Lin07,Luo-Che03}. For completeness we will include their short proofs in \cref{sec:app}.

Although there have been a large number of works devoted to determining the exact or asymptotic values for $A_q(n,d,w)$ and $A_q(n,d,\overline{w})$, most of them have rather strict restrictions on the set of parameters, e.g., $q,d,w,\overline{w}$ are some small and sporadic constants, see \cite{Lia-Ji19,Wei-Ge15,Che-Lin08,Ding08,Che-Lin07,Pol19,Wu-Fan09,Zha-Ge10,Zha-Ge13,Ge08,Wan21,Zha-Ge12}; for all fixed $w$ and $\overline{w}$ but $d\in\{2w-1,2w-2,2w-3\}$, see \cite{Che-Lin10,Che-Zha18,Che-Zha19}. In fact, as far as we can tell, in all of the previously known results the minimum distance $d$ was restricted to the above range.

We continue this line of research and determine the asymptotic values for $A_q(n,d,w)$ and $A_q(n,d,\overline{w})$ for a vast set of parameters. Our first result shows that for fixed odd distances, there exist CWCs asymptotically attaining the upper bound given by \cref{lem Ucwc}.

\begin{theorem}\label{thm:cwc}
Let $q$, $w$ and $d$ be fixed positive integers, where $d$ is odd. Let $t= \frac{2w-d+1}{2} $. Then for sufficiently large integer $n$, we have that
\begin{align*}
    A_q(n,d,w)\ge (1-o(1))\frac{(q-1)^t\binom{n}{t}}{\binom{w}{t}},
\end{align*}
where $o(1)\rightarrow 0$ as $n\rightarrow \infty$.
\end{theorem}

Combining Theorems \ref{lem Ucwc} and \ref{thm:cwc} we have that

\begin{proposition}
    Under the assumption of \cref{thm:cwc},
    \begin{align*}
        \lim_{n\rightarrow\infty}\frac{A_q(n,d,w)}{\binom{n}{t}}=\frac{(q-1)^t}{\binom{w}{t}}.
    \end{align*}
\end{proposition}

\noindent Note that the above result also implies the existence of approximate generalized Steiner systems, and solves an open problem of Etzion in an asymptotic form (see Problem 3.6 in \cite{Etz22} for more details).

Our second result shows that for fixed odd distances, there also exist CCCs asymptotically attaining the upper bound given by \cref{lem Uccc}.

\begin{theorem}\label{thm:ccc}
Let $q$, $w$ and $d$ be fixed positive integers, where $d$ is odd. Let $t=\frac{2w-d+1}{2}$. Then for given composition $\overline{w}=(w_1,\dots w_{q-1})$ with $w=\sum _{i=1}^{q-1}w_i$ and sufficiently large integer $n$, we have that
\begin{align*}
    A_q(n,d,\overline{w})\ge (1-o(1))\frac{\binom{n}{n-w,w_1,\dots,w_{q-1}}}{\binom{n-t}{w-t}f(\overline{w},w-t)}.
\end{align*}
where $o(1)\rightarrow 0$ as $n\rightarrow \infty$.
\end{theorem}

Combining Theorems \ref{lem Uccc} and \ref{thm:ccc}, we have that

\begin{proposition}
    Under the assumption of \cref{thm:ccc},
    \begin{align*}
        \lim_{n\rightarrow\infty}\frac{\binom{n-t}{w-t}A_q(n,d,\overline{w})}{\binom{n}{n-w,w_1,\dots,w_{q-1}}}=\frac{1}{f(\overline{w},w-t)}.
    \end{align*}
\end{proposition}

\paragraph{Outline of the paper.} The remaining part of this paper is organized as follows. In \cref{sec:matchinges}, we will illustrate the connection between codes and hypergraphs, and then introduce the results of \cite{Kah96,Del-Pos22,Glo-Lic23} on the existence of near-optimal matchings in hypergraphs. The proofs of Theorems \ref{thm:cwc} and \ref{thm:ccc} will be presented in Sections \ref{sec:cwc} and \ref{sec:ccc}, respectively. (Note that we present proof overviews of the two theorems at the beginning of Sections \ref{sec:cwc} and \ref{sec:ccc}.) We will conclude the paper with some open questions in \cref{sec:con}.

\section{Codes and matchings in hypergraphs}\label{sec:matchinges}

\subsection{Constant weight codes, constant composition codes, and hypergraphs}

\noindent The main task of this subsection is to establish a connection between CWCs (resp. CCCs) and hypergraphs.

Let $J_q(n,w):=\{\mathbf{x}\in\Sigma_q^n:\wt(\mathbf{x})=w\}$ denote the set formed by all elements in $\Sigma_q^n$ with Hamming weight $w$. Similarly, given a composition $\overline{w}=(w_1,\ldots,w_{q-1})$, let
$$J_q(n,\overline{w}):=\{\mathbf{x}\in \Sigma_q^n:\mathbf{x}~\text{has~composition}~\overline{w}\}.$$
%let $J_q(n,\overline{w}):=\{\mathbf{x}\in \Sigma_q^n:|\{j\in [n]:x_j=i\}|=w_i \text{ for every } i\in [q-1]\}$ denote the set formed by all elements in $\Sigma_q^n$ with composition $\overline{w}$.
It is easy to see that $|J_q(n,w)|=\bi{n}{w}(q-1)^w$ and
\begin{align*}
   |J_q(n,\overline{w})|=\binom{n}{w}\binom{w}{w_1}\binom{w-w_1}{w_2}\cdots\binom{w-w_1-\cdots-w_{q-2}}{w_{q-1}}=\binom{n}{n-w,w_1,\ldots,w_{q-1}}.
\end{align*}

A hypergraph $\ma{H}$ can be viewed as a pair $\ma{H}=(V(\ma{H}),E(\ma{H}))$, where the vertex set $V(\ma{H})$ is a finite set and the edge set $E(\ma{H})$ is a collection of subsets of $V(\ma{H})$. %By slightly abuse of notation we will also use $\mathcal{H}$ to represent $E(\mathcal{H})$.
A hypergraph is {\it $l$-bounded} if every edge contains at most $l$ vertices; moreover, it is {\it $w$-uniform} if every edge contains exactly $w$ vertices. A hypergraph $\ma{H}$ is {\it $n$-partite} if its vertex set $V(\ma{H})$ admits a partition $V(\ma{H})=\cup_{i=1}^n V_i$ such that every edge of $\mathcal{H}$ has at most one intersection with each $V_i$.

\begin{definition}[Hypergraph associated with $J_q(n,w)$]\label{def:hypergraph-cwc}
For positive integers $q,n,w$ with $w\le n$, let $\mathcal{J}_q(n,w)$ denote the complete $n$-partite $w$-uniform hypergraph with equal part size $q-1$, where we assume without loss of generality that $V(\mathcal{J}_q(n,w))$ admits a partition $V(\mathcal{J}_q(n,w))=\cup_{i=1}^n V_i$ such that for each $i\in[n]$, $V_i=\{(i,a):a\in [q-1]\}$. With this notation,
$$E(\mathcal{J}_q(n,w))=\big\{\{(i_1,a_{i_1}),\ldots,(i_w,a_{i_w})\}:1\le i_1<\cdots<i_w\le n,a_{i_1},\ldots,a_{i_w}\in [q-1]\}\big\}.$$

\noindent Clearly, $|V(\mathcal{J}_q(n,w))|=n(q-1)$ and $|E(\mathcal{J}_q(n,w))|=\bi{n}{w}(q-1)^w$.
\end{definition}

For a vector $\mathbf{x}=(x_1,\ldots,x_n)\in\Sigma_q^n$, the {\it support} of $\mathbf{x}$ is denoted by $\supp(\mathbf{x}):=\{i\in[n]:x_i\neq 0\}$. Consider the mapping $\pi:J_q(n,w)\rightarrow\mathcal{J}_q(n,w)$ which sends a vector $$\mathbf{x}=(x_1,\ldots,x_n)\in J_q(n,w)$$ to an edge
\begin{align}\label{eq:mapping}
    \pi(\mathbf{x}):=\{(i,x_i):i\in \supp(x)\}\in \mathcal{J}_q(n,w).
\end{align}
It is routine to check that $\pi$ is a bijection.

For $e\in\mathcal{J}_q(n,w)$, denote $\supp(e):=\supp(\pi^{-1}(e))$, which is precisely the subset of $[n]$ recording the vertex parts $V_i,~i\in [n]$ such that $V_i\cap e\neq\emptyset$. More generally, for every $f\subseteq\cup_{i=1}^n V_i$ with $|f\cap V_i|\le 1$ for each $i\in [n]$, define
\begin{align}\label{eq:support}
    \supp(f):=\{i\in[n]:|f\cap V_i|=1\}.
\end{align}

For a code $C\subseteq J_q(n,w)$, let $\pi(C)=\{\pi(\mathbf{x}):\mathbf{x}\in C\}\subseteq\mathcal{J}_q(n,w)$. It is not hard to check that for $\mathbf{x},\mathbf{y}\in J_q(n,w)$,
\begin{align}\label{eq:distance}
    d(\mathbf{x},\mathbf{y})=2w-|\pi(\mathbf{x})\cap\pi(\mathbf{y})|-|\supp(\mathbf{x})\cap\supp(\mathbf{y})|.
\end{align}

\iffalse
\begin{figure}
    \centering
    \scalebox{0.6}{\begin{tikzpicture}[domain=-3:1]
    \draw (0,0) rectangle (10,1);
    \node at (0,.5) [left] {$\mathbf{x}$};
    % \draw[decorate,decoration={brace,raise=8pt,yellow!50}](0,1) -- (2,1)
    \fill[black!80] (0,0) rectangle (2,1);
    \fill[black!40] (2,0) rectangle (5,1);
    \fill[black!40] (7,0) rectangle (7.5,1);
    \fill[black!40] (8.5,0) rectangle (9,1);

    \draw (0,-1) rectangle (10,-2);
    \node at (0,-1.5) [left] {$\mathbf{y}$};
    \fill[black!80] (0,-1) rectangle (2,-2);
    \fill[black!40] (2,-1) rectangle (5,-2);
    \fill[black!40] (6,-1) rectangle (6.5,-2);
    \fill[black!40] (9.2,-1) rectangle (9.7,-2);
    \end{tikzpicture}}
    \label{fig:minimum distance}
    \caption{$d(\mathbf{x},\mathbf{y})$}
\end{figure}
\fi

We frequently will make use of the equations \eqref{eq:mapping}, \eqref{eq:support}, and \eqref{eq:distance}. Let us conclude this subsection with the following definition.

\begin{definition}[Hypergraph associated with $J_q(n,\overline{w})$]\label{def:hypergraph-ccc}
    Note that $J_q(n,\overline{w})\subseteq J_q(n,w)$. Let $\mathcal{J}_q(n,\overline{w}):=\pi(J_q(n,\overline{w}))\subseteq \mathcal{J}_q(n,w)$ be the image of $J_q(n,\overline{w})$ given by $\pi$. Then $\mathcal{J}_q(n,\overline{w})$ is an $n$-partite $w$-uniform hypergraph with vertex set $V(\mathcal{J}_q(n,\overline{w}))=V(\mathcal{J}_q(n,w))=\cup_{i=1}^n V_i$, where the $V_i$'s are given in \cref{def:hypergraph-cwc}, and edge set $E(\mathcal{J}_q(n,\overline{w}))=\{\pi(\mathbf{x}):\mathbf{x}\in J_q(n,\overline{w})\}$. Clearly, $|V(\mathcal{J}_q(n,\overline{w}))|=nq$ and $|E(\mathcal{J}_q(n,\overline{w}))|=|J_q(n,\overline{w})|=\binom{n}{n-w,w_1,\dots,w_{q-1}}$.
\end{definition}

\subsection{Kahn's theorem on the existence of near-optimal matchinges}

\noindent For a hypergraph $\ma{H}$ and a vertex $v\in V(\ma{H})$, the {\it degree} $\deg (v)$ of $v$ is the number of edges in $E(\ma{H})$ that contain $v$. Let $\Delta(\ma{H})$ be the {\it maximum degree} of $\ma{H}$, i.e., $\Delta(\ma{H})=\max_{v\in V(\ma{H})}\deg(v)$. For distinct vertices $v_1,v_2\in V(\ma{H})$, let $\deg(v_1,v_2)$ be the number of edges in $E(\ma{H})$ that contain both $v_1$ and $v_2$. Let $\cod (\ma{H})$ be the {\it maximum codegree} of $\ma{H}$, i.e., $\cod(\ma{H})=\max_{v_1\neq v_2\in V(\ma{H})}\deg(v_1,v_2).$

A {\it matching} of $\ma{H}$ is a set of pairwise disjoint edges of $E(\ma{H})$. The {\it matching number} of $\ma{H}$, denoted by $\nu(\ma{H})$, is the size of the maximum matching of $\ma{H}$. We say that a function $f:E(\mathcal{H})\rightarrow [0,1]$ is a {\it fractional matching} of $\mathcal{H}$ if for every $v\in V(\mathcal{H})$, $$\sum_{e\in E(\mathcal{H}):~v\in e}f(e)\le 1.$$ For a fractional matching $f$ of $\mathcal{H}$, let $f(\mathcal{H}):=\sum_{e\in E(\mathcal{H})}f(e)$ and
$$\alpha(f):=\max_{x\neq y \in V(\mathcal{H})} \sum_{e\in E(\mathcal{H}):~x,y\in e} f(e).$$
\iffalse
The {\it fractional matching number} of $\mathcal{H}$ is defined as
\begin{align*}
    \nu^*(\mathcal{H}):=\max\{f(\mathcal{H}):f~\text{is a fractional matching of}~\mathcal{H}\}.
\end{align*}
Since every matching naturally defines a fractional matching, we always have $\nu^*(\mathcal{H})\ge\nu(\ma{H})$.
\fi
The following result of Kahn showed that in certain case $\nu(\mathcal{H})$ can be lower bounded by $f(\mathcal{H})$.

\begin{lemma}[Kahn's theorem, see Theorem 1.2 in \cite{Kah96}]\label{thm:Kahn}
    For every integer $l\ge 1$ and every real $\epsilon>0$, there exists a real $\sigma$ such that whenever $\ma{H}$ is a $l$-bounded hypergraph and $f$ is a fractional matching of $\ma{H}$ with $\alpha(f)<\sigma$, then $\nu(\mathcal{H})>(1-\epsilon)f(\mathcal{H}).$
\end{lemma}

To prove \cref{thm:cwc}, we will make use of the following consequence of \cref{thm:Kahn}, which shows that every hypergraph with large maximum degree and small codegree has a relatively large matching number.

\begin{lemma}\label{cor:Kahn}
    Let $\mathcal{H}$ be a $l$-bounded hypergraph with $\cod(\mathcal{H})/\Delta(\mathcal{H})=o(1)$, where $o(1)\rightarrow 0$ as $|V(\mathcal{H})|\rightarrow\infty$. Then we have $\nu(\mathcal{H})>(1-o(1))|E(\mathcal{H})|/\Delta(\mathcal{H})$.
\end{lemma}

\begin{proof}
It is easy to see that the constant function $f_c\equiv 1/\Delta(\mathcal{H})$ is a valid fractional matching of $\mathcal{H}$. Moreover, we have that $f_c(\mathcal{H})=|E(\mathcal{H})|/\Delta(\mathcal{H})$ and
\begin{align*}
\alpha(f_c)=\max_{x\neq y\in V(\mathcal{H})} \frac{\deg (x,y)}{\Delta(\mathcal{H})}=\frac{\cod(\mathcal{H})}{\Delta(\mathcal{H})}.
\end{align*}
Then the conclusion follows directly by applying \cref{thm:Kahn} with $f:=f_c$.
\end{proof}

\subsection{Near-optimal matchinges avoiding forbidden configurations}

\noindent Let $\mathcal{H}$ be a hypergraph. A {\it configuration hypergraph} for $\mathcal{H}$ is a hypergraph $\mathcal{F}$ with vertex set $V(\mathcal{F})=E(\mathcal{H})$ and edge set $E(\mathcal{F})$ formed by some subsets of $E(\mathcal{H})$.
We say that a matching of $\mathcal{H}$ is {\it $\mathcal{F}$-free} if it contains no edge in $E(\mathcal{F})$ as a subset. Note that $E(\mathcal{F})$ is defined by the set of ``bad'' configurations that we want to avoid; in particular, in the proof of \cref{thm:ccc} $E(\mathcal{F})$ will be defined by the pairs of vectors in $J_q(n,\overline{w})$ that violate the minimum distance condition of a CCC. \cref{thm:conflict-free} below presents several sufficient conditions on which $\mathcal{H}$ and $\mathcal{F}$ should satisfy so that one can find a large $\mathcal{F}$-free matching in $\mathcal{H}$.

We will need some necessary definitions before presenting \cref{thm:conflict-free}. Note that all of them can be found in \cite{Del-Pos22}.

\begin{definition}[see \cite{Del-Pos22}]
    Let $\mathcal{H}$ be a $k$-bounded hypergraph and $\mathcal{F}$ be a $l$-bounded configuration hypergraph of $\mathcal{H}$.
    \begin{itemize}
        \item For every $i\in[k]$, the {\it $i$-degree} of a vertex $v\in V(\mathcal{H})$ is defined as the number of edges in $E(\mathcal{H})$ of size $i$ that contain $v$. The {\it maximum $i$-degree} of $\mathcal{H}$, denoted as $\Delta_i(\mathcal{H})$, is the maximum $i$-degree over all $v\in V(\mathcal{H})$, i.e.,
\begin{equation*}
    \Delta_i(\mathcal{H})=\max\limits_{v\in V(\mathcal{H})}|\{e\in E(\mathcal{H}):v\in e,~|e|=i\}|.
\end{equation*}

        \item For every $2\le j<r\le l$, the {\it maximum $(r,j)$-codegree} of $\mathcal{F}$ is defined as
        $$\Delta_{r,j}(\mathcal{F}):=\max\limits_{S\in \binom{V(\mathcal{F})}{j}}|\{e\in E(\mathcal{F}):S\subseteq e,~|e|=r\}|. $$

        \item The {\it $2$-codegree} of $v\in V(\mathcal{H})$ and $e\in V(\mathcal{F})=E(\mathcal{H})$ with $v\notin e$ is the number of edges in $E(\mathcal{F})$ of size $2$ that contains both $e$ and some $e'\in V(\mathcal{F})$ with $v\in e'$. The {\it maximum $2$-codegree} of $\mathcal{H}$ with respect to $\mathcal{F}$ is the maximum $2$-codegree over all $v\in V(\mathcal{H})$ and $e\in V(\mathcal{F})$ with $v\notin e$.

        \item The {\it common $2$-degree} of distinct $e,e'\in V(\mathcal{F})$ is the number of $w\in V(\mathcal{F})$ such that both $\{e,w\}$ and $\{e',w\}$ are edges of $\mathcal{F}$.
        %form an edge with $u$ and $v$, i.e. $|\{w\in V(\mathcal{F}):\{u,w\},\{w,v\}\in E(\mathcal{F})\}|$. Similarly,
        The {\it maximum common $2$-degree} of $\mathcal{F}$ is the maximum common $2$-degree over all distinct $e,e'\in V(\mathcal{F})$.
    \end{itemize}
\end{definition}

\begin{lemma}[see Corollary 1.17 in \cite{Del-Pos22}, also Theorem 1.3 in \cite{Glo-Lic23}]\label{thm:conflict-free}
    For all integers $k,l\ge 2$, and real $\beta> 0$, there exist an integer $D_{\beta}\ge 0$ and a real $\alpha> 0$ such that following holds for all $D\ge D_{\beta}$.

    Let $\mathcal{H}$ be an $k$-bounded hypergraph with $\Delta(\mathcal{H})\le D$ and $\cod(\mathcal{H})\le D^{1-\beta}$. Let $\mathcal{F}$ be a $l$-bounded configuration hypergraph of $\mathcal{H}$. Suppose that the following conditions hold:
    \begin{enumerate}
        \item [{\rm (1)}] $\Delta_{i}(\mathcal{F})\le \alpha \cdot D^{i-1}\log D$ for all $2\le i\le l$;

        \item [{\rm (2)}] $\Delta_{r,j}(\mathcal{F})\le D^{r-j-\beta}$ for all $2\le j<r\le l$;

        \item [{\rm (3)}] The maximum $2$-codegree of $\mathcal{H}$ with respect to $\mathcal{F}$ is at most $D^{1-\beta}$;

        \item [{\rm (4)}] The maximum common $2$-degree of $\mathcal{F}$ is at most $D^{1-\beta}$.
    \end{enumerate}
  Then there exists an $\mathcal{F}$-free matching of $\mathcal{H}$ with size at least $\frac{|E(\mathcal{H})|}{D}\cdot (1-D^{-\alpha})$.
\end{lemma}

\section{Proof of \cref{thm:cwc} via Kahn's theorem}\label{sec:cwc}

\paragraph{Proof overview.} We will prove \cref{thm:cwc} via \cref{thm:Kahn}. The rough idea is to make use of the connection between CWCs and hypergraphs built in \eqref{eq:mapping}, \eqref{eq:support}, and \eqref{eq:distance}. Then the problem of constructing a large CWC can be converted to the problem of finding a large matching in an auxiliary hypergraph $\mathcal{H}$ defined as below (see \cref{lem:matching-cwc} for details). The latter problem can be solved by applying \cref{thm:Kahn} (see \cref{lem:large-matching-via-Kahn} for details).

\vspace{10pt}

To proceed, recall the hypergraph $\mathcal{J}_q(n,w)$ defined in \cref{def:hypergraph-cwc}. We will first construct a hypergraph $\mathcal{G}$ with vertex set $V(\mathcal{G})= V(\mathcal{J}_q(n,w))\cup [n]$ and edge set
\begin{align*}
E(\mathcal{G})=\{e\cup \supp(e):e\in\mathcal{J}_q(n,w)\}.
\end{align*}
As for every $e\in\mathcal{J}_q(n,w)$, $|e|=|\supp(e)|=w$, $\mathcal{G}$ is a $2w$-uniform hypergraph with $|V(\mathcal{G})|=nq$ and $|E(\mathcal{G})|=|E(\mathcal{J}_q(n,w))|=\binom{n}{w}(q-1)^w$. Let $\mathcal{H}$ be the hypergraph defined as
\begin{align*}
V(\mathcal{H})=\binom{V(\mathcal{G})}{2t}=\binom{(\cup_{i=1}^n V_i)\cup [n]}{2t}  \end{align*}
and
\begin{align*}
E(\mathcal{H})=\left\{\binom{e'}{2t}:e'\in E(\mathcal{G})\right\}=\left\{\binom{e\cup\supp(e)}{2t}:e\in\mathcal{J}_q(n,w)\right\}.
\end{align*}
Hence, $\mathcal{H}$ is a $\binom{2w}{2t}$-uniform hypergraph with $|E(\mathcal{H})|=\binom{n}{w}(q-1)^w$.

The next lemma shows that every matching of $\mathcal{H}$ defines an $(n,d,w)_q$-CWC.

\begin{lemma}\label{lem:matching-cwc}
    Every matching $\mathcal{M}$ of $\mathcal{H}$ defines an $(n,d,w)_q$-CWC $C_{\mathcal{M}}\subseteq J_q(n,w)$ with $|C_{\mathcal{M}}|=|\mathcal{M}|$.
\end{lemma}

\begin{proof}
    Consider a matching $$\mathcal{M}=\left\{\binom{e^i\cup\supp(e^i)}{2t}:i\in [m]\right\}\subseteq\mathcal{H}$$
    of size $m$, where for every $i$, $e^i\in\mathcal{J}_q(n,w)$. Let $x^i:=\pi^{-1}(e^i)$, where $\pi$ is the mapping defined in \eqref{eq:mapping}, and let $C_{\mathcal{M}}:=\{x^i:i\in[m]\}\subseteq J_q(n,w)$. To prove the lemma, it remains to show that for all distinct $i,j\in[m]$, $d(x^i,x^j)\ge d$. Indeed, by \eqref{eq:distance} we have that
    \begin{align*}
    d(x^i,x^j)=2w-|\pi(x^i)\cap\pi(x^j)|-|\supp(x^i)\cap\supp(x^j)|=2w-|e^i\cap e^j|-|\supp(e^i)\cap\supp(e^j)|.
    \end{align*}
    Since $\mathcal{M}$ is a matching in $\mathcal{H}$, we have that $$\binom{e^i\cup\supp(e^i)}{2t}\bigcap\binom{e^j\cup\supp(e^j)}{2t}=\emptyset,$$ which implies that $|e^i\cap e^j|+|\supp(e^i)\cap\supp(e^j)|\le 2t-1$. It follows that
    \begin{align*}
        d(x^i,x^j)\ge2w-(2t-1)=d,
    \end{align*}
    completing the proof of the lemma.
\end{proof}

The lemma below shows that $\mathcal{H}$ indeed has a large matching number.

\begin{lemma}\label{lem:large-matching-via-Kahn}
    %\sout{$\mathcal{H}$ has a matching with size at least $(1-o(1))(q-1)^t\frac{\binom{n}{t}}{\binom{w}{t}}$.}
    $\nu(\mathcal{H})\ge(1-o(1))\frac{(q-1)^t\binom{n}{t}}{\binom{w}{t}}$, where $o(1)\rightarrow 0$ as $n\rightarrow\infty$.
\end{lemma}

\begin{proof}
We will prove the lemma by applying \cref{cor:Kahn} to $\mathcal{H}$. As $|E(\mathcal{H})|=\binom{n}{w}(q-1)^w$, to prove the lemma, it suffices to show that (i) $\Delta(\mathcal{H})=\binom{n-t}{w-t}(q-1)^{w-t}$ and (ii) $\cod(\mathcal{H})/\Delta(\mathcal{H})=o(1)$, where $o(1)\rightarrow 0$ as $n\rightarrow\infty$.

To prove (i), consider an arbitrary vertex $$u\in V(\mathcal{H})=\binom{(\cup_{i=1}^n V_i)\cup [n]}{2t}.$$
Then
$$\deg(u)=|\{e\in\mathcal{J}_q(n,w):u\subseteq e\cup\supp(e)\}|.$$
As $\mathcal{J}_q(n,w)$ is $n$-partite, i.e., $|e\cap V_i|\le 1$ for every $e\in\mathcal{J}_q(n,w)$ and $i\in [n]$, to prove (i) we can assume without loss of generality that for every $i\in[n]$, $|u\cap V_i|\le 1$, since otherwise $\deg(u)=0$. Let $f=u\cap(\cup_{i=1}^n V_i)$ and recall the definition of $\supp(f)$ in \eqref{eq:support}. Clearly, $|\supp(f)|=|f|$ and
\begin{align}\label{eq:sum=2t}
    |u|=|\supp(f)|+|u\cap [n]|=2t.
\end{align}
We will consider the following two cases.
\begin{itemize}
    \item [(a)] If $\supp(f)\neq u\cap [n]$, then
    \begin{align*}
        |\supp(f)\cap(u\cap [n])|<\min\{|\supp(f)|,|u\cap [n]|\}=\min\{|\supp(f)|,2t-|\supp(f)|\}\le t.
    \end{align*}
    It follows that
    \begin{align*}
        |\supp(f)\cup(u\cap [n])|=|\supp(f)|+|u\cap [n]|-|\supp(f)\cap(u\cap [n])|\ge t+1.
    \end{align*}
    Therefore,
    \begin{align*}
\deg(u)%&=|\{e\in\mathcal{J}_q(n,w):u\subseteq e\cup\supp(e)\}|\\
&\le|\{e\in\mathcal{J}_q(n,w):\supp(f)\cup(u\cap [n])\subseteq\supp(e)\}|\\
&\le\binom{n-t-1}{w-t-1}(q-1)^w,
\end{align*}
where the first inequality holds since for every $e\in\mathcal{J}_q(n,w)$ with $u\subseteq e\cup\supp(e)$ we must have $\supp(f)\subseteq\supp(e)$ and $u\cap[n]\subseteq\supp(e)$, and the second inequality holds since given that  $\supp(f)\cup(u\cap [n])\subseteq\supp(e)$ and $|\supp(f)\cup(u\cap [n])|\ge t+1$ there are at most $\binom{n-t-1}{w-t-1}$ choice for $\supp(e)$, and for every fixed $\supp(e)$, there are at most $(q-1)^w$ choices for $e$.

\item [(b)] If $\supp(f)=u\cap [n]$, then by \eqref{eq:sum=2t} we further have that $|\supp(f)|=|u\cap [n]|=t$, and hence
\begin{align*}
   \deg(u)&=|\{e\in\mathcal{J}_q(n,w):u\subseteq e\cup\supp(e)\}|\\
   &=\binom{n-t}{w-t}(q-1)^{w-t},
\end{align*}
where the second equality holds since under assumption (b), there are $\binom{n-t}{w-t}$ choices for $\supp(e)$, and for every fixed $\supp(e)$, there are $(q-1)^{w-t}$ choices for $e\setminus f$.
\end{itemize}
As $w,t,q$ are all fixed, for sufficiently large $n$ we have $\Delta(\mathcal{H})=\binom{n-t}{w-t}(q-1)^{w-t}$, as needed.

It remains to prove (ii). Consider arbitrary two distinct vertices $u^1,u^2\in V(\mathcal{H})$. Then
\begin{align*}
    \deg(u^1,u^2)=|\{e\in\mathcal{J}_q(n,w):u^1\subseteq e\cup\supp(e)~\text{and}~u^2\subseteq e\cup\supp(e)\}|.
\end{align*}
Let $f^1=u^1\cap(\cup_{i=1}^n V_i)$ and $f^2=u^2\cap(\cup_{i=1}^n V_i)$. As $\deg(u^1,u^2)\le\min\{\deg(u^1),\deg(u^2)\}$, to prove (ii) we only need to consider the case in which
\begin{align}\label{eq:assumption}
    \text{$\supp(f^1)=u^1\cap [n]$  and $\supp(f^2)=u^2\cap [n]$,}
\end{align}
 since otherwise similarly to the proof in (a) one can show that $$\deg(u^1,u^2)\le\min\{\deg(u^1),\deg(u^2)\}\le\binom{n-t-1}{w-t-1}(q-1)^w=o(\Delta(\mathcal{H})).$$
 Given \eqref{eq:assumption}, it follows by \eqref{eq:sum=2t} that $|\supp(f^1)|=|u^1\cap [n]|=t$ and $|\supp(f^2)|=|u^2\cap [n]|=t$. The remaining proof is divided into two cases.
\begin{itemize}
    \item [(c)] If $\supp(f^1)\neq\supp(f^2)$, then $|\supp(f^1)\cup\supp(f^2)|\ge t+1$. Again similarly to the proof in (a), one can show that
    \begin{align*}
        \deg(u^1,u^2)&\le|\{e\in\mathcal{J}_q(n,w):\supp(f^1)\cup\supp(f^2)\subseteq\supp(e)\}|\\
        &\le\binom{n-t-1}{w-t-1}(q-1)^w=o(\Delta(\mathcal{H})).
    \end{align*}

    \item [(d)] Suppose that $\supp(f^1)=\supp(f^2)$. If $f^1\neq f^2$, then by the $n$-partiteness of $\mathcal{J}_q(n,w)$ we have
    \begin{align*}
        \deg(u^1,u^2)\le|\{e\in\mathcal{J}_q(n,w):f^1\cup f^2\subseteq e\}|=0.
    \end{align*}
    If $f^1=f^2$, then we have $u^1\cap(\cup_{i=1}^n V_i)=u^2\cap(\cup_{i=1}^n V_i)$, which together with \eqref{eq:assumption} imply that $u^1=u^2$, a contradiction.
\end{itemize}
The proof of (ii) is also completed.
\end{proof}

\begin{proof}[Proof of \cref{thm:cwc}]
    \cref{thm:cwc} is a direct consequence of Lemmas \ref{lem:matching-cwc} and \ref{lem:large-matching-via-Kahn}.
\end{proof}

\section{Proof of \cref{thm:ccc} via $\mathcal{F}$-free matchinges}\label{sec:ccc}

\paragraph{Proof overview.} We will prove \cref{thm:ccc} via \cref{thm:conflict-free}. The high-level idea is similar to that in the proof of \cref{thm:cwc}. We will appropriately define an auxiliary hypergraph $\mathcal{H}$ and a corresponding configuration hypergraph $\mathcal{F}$ for $\mathcal{H}$, and then convert the problem of constructing a large CCC to the problem of finding a large $\mathcal{F}$-free matching in $\mathcal{H}$ (see \cref{lem:matching-ccc} for details). The latter problem can be solved by applying \cref{thm:conflict-free} (see \cref{lem:large-matching-via-Conflict-free} for details).

\vspace{10pt}

Recall the hypergraph $\mathcal{J}_q(n,\overline{w})$ defined in \cref{def:hypergraph-ccc}. To prove \cref{thm:ccc} let us construct an auxiliary hypergraph $\mathcal{H}$ with $$V(\mathcal{H})=\binom{V(\mathcal{J}_q(n,\overline{w}))}{t}=\binom{\cup_{i=1}^n V_i}{t}$$ and $$E(\mathcal{H})=\left\{ \binom{e}{t}:e\in \mathcal{J}_q(n,\overline{w})\right\}.$$ It is clear that $\mathcal{H}$ is a $\binom{w}{t}$-uniform hypergraph with $|V(\mathcal{H})|=\binom{nq}{t}$ and $|E(\mathcal{H})|=\binom{n}{n-w,w_1,\dots,w_{q-1}}$.

Let $\mathcal{F}$ be the configuration hypergraph for $\mathcal{H}$ with $V(\mathcal{F})=E(\mathcal{H})$ and
\begin{align*}
   E(\mathcal{F})=\left\{\left\{\binom{e^1}{t},\binom{e^2}{t}\right\}:e^1,e^2\in \mathcal{J}_q(n,\overline{w}),~|\supp(e^1)\cap \supp(e^2)|>t\right\}.
\end{align*}

The following lemma shows that every $\mathcal{F}$-free matching of $\mathcal{H}$ defines an $(n,d,\overline{w})_q$-CCC.

\begin{lemma}\label{lem:matching-ccc}
    Every $\mathcal{F}$-free matching $\mathcal{M}$ of $\mathcal{H}$ defines an $(n,d,\overline{w})_q$-CCC $C_{\mathcal{M}}\subseteq J_q(n,\overline{w})$ with $|C_{\mathcal{M}}|=|\mathcal{M}|$.
\end{lemma}

\begin{proof}
    Let $\mathcal{M}=\left\{\binom{e^i}{t}:i\in [m]\right\}$ be a matching of $\mathcal{H}$ with size $m$, where for every $i$, $e^i\in \mathcal{J}_q(n,\overline{w})$. Let $x^i:=\pi^{-1}(e^i)$, where $\pi$ is the mapping defined in \eqref{eq:mapping}, and let $C_{\mathcal{M}}:=\{x^i:i\in[m]\}\subseteq J_q(n,\overline{w})$. To prove the lemma, it remains to show that for all distinct $i,j\in [m]$, $d(x^i,x^j)\ge d$. By \eqref{eq:distance}, we have that
    $$d(x^i,x^j)=2w-|\pi(x^i)\cap\pi(x^j)|-|\supp(x^i)\cap\supp(x^j)|=2w-|e^i\cap e^j|-|\supp(e^i)\cap\supp(e^j)|. $$
    Since $\mathcal{M}$ is $\mathcal{F}$-free, we have that $$|\supp(e^i)\cap \supp(e^j)|\le t.$$ Moreover, since $\mathcal{M}$ is a matching of $\mathcal{H}$, it follows that $$\binom{e^i}{t}\bigcap \binom{e^j}{t}=\emptyset,$$ which implies that $|e^i\cap e^j|\le t-1$. Hence $$d(x^i,x^j)\ge 2w-(t-1)-t=d,$$ completing the proof of the lemma.
\end{proof}

The lemma below shows that $\mathcal{H}$ indeed has a large $\mathcal{F}$-free matching.

\begin{lemma}\label{lem:large-matching-via-Conflict-free}
     $\mathcal{H}$ has a $\mathcal{F}$-free matching with size at least $(1-n^{t-w})\frac{\binom{n}{n-w,w_1,\dots,w_{q-1}}}{\binom{n-t}{w-t}f(\overline{w},w-t)}$ when $n$ is sufficiently large.
\end{lemma}

\begin{proof}
    We will prove the lemma by applying \cref{thm:conflict-free}. Below we will verify that $\mathcal{H}$ and $\mathcal{F}$ satisfy the assumptions of \cref{thm:conflict-free}. Let $\beta:=\frac{1}{2(w-t)}$ and $D:=\binom{n-t}{w-t}f(\overline{w},w-t)$ and $\alpha =1$.

\begin{claim}\label{claim-1}
    $\mathcal{H}$ satisfies (i) $\Delta(\mathcal{H})\le D$ and (ii) $\cod(\mathcal{H})=O(D^{1-\beta})$.
\end{claim}
    To prove (i), we associate every vertex $u\in V(\mathcal{H})=\binom{\cup_{i=1}^n V_i}{t}$ with a vector $\overline{u}=(u_1,\dots,u_{q-1})$, where for every $j\in [q-1]$, $u_j:=|\{i\in [n]:(i,j)\in u\}|$. Clearly, $\sum_{j=1}^n u_j=t$. It is routine to check by the definition of $\mathcal{H}$ and $\mathcal{J}_q(n,\overline{w})$ that $\deg(u)>0$ is only possible when $\overline{u}$ is $(\overline{w},t)$-admissible and $|u\cap V_i|\le 1$ for every $i\in [n]$. Moreover, for every such $u$ we have that
    \begin{align*}
        \deg(u)&=|\{e\in\mathcal{J}_q(n,\overline{w}):u\subseteq e\}|\\
        &=\binom{n-t}{w-t}\binom{w-t}{w_1-u_1,\dots,w_{q-1}-u_{q-1}}\\
        &\le\binom{n-t}{w-t}f(\overline{w},w-t)=D,
    \end{align*}
    where the second equality holds since given $u\subseteq e$ there are $\binom{n-t}{w-t}$ choices for $\supp(e)$ and for every fixed $\supp(e)$, there are $\binom{w-t}{w_1-u_1,\dots,w_{q-1}-u_{q-1}}$ choices for $e\setminus u$; the inequality follows by the definition of $f(\overline{w},w-t)$.

    To prove (ii), note that for every distinct $u^1,u^2\in V(\mathcal{H})$, $$\deg(u^1,u^2)=|\{e\in \mathcal{J}_q(n,\overline{w}):u^1,u^2\subseteq e\}|\le \min\{\deg_{\mathcal{H}}(u^1),\deg_{\mathcal{H}}(u^2)\}. $$
    If $\supp(u^1)=\supp(u^2)$, then since $u^1\neq u^2$, there exists $i\in \supp(u^1)=\supp(u^2)$ such that $|(u^1\cup u^2)\cap V_i|\ge 2$. Then it is clear that in this case $\deg(u^1,u^2)=0$. If $\supp(u^1)\neq \supp(u^2)$, then $|\supp(u^1)\cup \supp(u^2)|\ge t+1$, and hence
    \begin{align*}
        \deg(u^1,u^2)&\le |\{e\in\mathcal{J}_q(n,\overline{w}):\supp(u^1)\cup \supp(u^2)\subseteq \supp(e)\}|\\
        &\le \binom{n-t-1}{w-t-1}\binom{w}{w_1,\dots,w_{q-1}},
    \end{align*}
    where the second inequality holds since given that $\supp(u^1)\cup \supp(u^2)\subseteq \supp(e)$, we have at most $\binom{n-t-1}{w-t-1}$ choices for $\supp(e)$ and for every fixed $\supp(e)$, there are $\binom{w}{w_1,\dots,w_{q-1}}$ choices for $e$. Hence, for fixed $\overline{w},t,q$ and sufficiently large $n$, we have $$\cod (\mathcal{H})\le \binom{n-t-1}{w-t-1}\binom{w}{w_1,\dots,w_{q-1}}\le D^{1-\beta},$$
    as needed.

    \begin{claim}\label{claim-2}
        $\mathcal{F}$ satisfies all of the assumptions (1)-(4) listed in \cref{thm:conflict-free}.
    \end{claim}

    Note that $\mathcal{F}$ is a $2$-uniform hypergraph (i.e., a graph) with $V(\mathcal{F})=E(\mathcal{H})$. So (2) holds vacuously. To verify (1), by the definition of $\mathcal{F}$, we have
    \begin{align*}
        \Delta_2(\mathcal{F})& = \max_{\binom{e^1}{t}\in V(\mathcal{F})}\bigg|\left\{\left\{\binom{e^1}{t},{e^2\choose t}\right\}\in E(\mathcal{F}):e^2\in \mathcal{J}_q(n,\overline{w})\right\}\bigg|\\
        &= \max_{e^1\in \mathcal{J}_q(n,\overline{w})}\mid \{e^2\in \mathcal{J}_q(n,\overline{w}):|\supp(e^1)\cap \supp(e^2)|>t\}\mid\\
        &\le \sum_{i=t+1}^{w}\binom{w}{i}\binom{n-w}{w-i}\binom{w}{w_1,\dots,w_{q-1}}\\
        % &\le (w-t)\binom{w}{t+1}\binom{n-w}{w-t-1}\binom{w}{w_1,\dots,w_{q-1}}\\
        &\le D^{1-\beta},
    \end{align*}
    where the first inequality holds since for every fixed $t+1\le i\le w$, there are $\binom{w}{i}$ choices for $\supp(e_1)\cap \supp(e_2)$, for fixed $\supp(e_1)\cap \supp(e_2)$, there are $\binom{n-w}{w-i}$ choices for $\supp(e_2)$ and for fixed $\supp(e^2)$, there are at most $\binom{w}{w_1,\dots,w_{q-1}}$ choices for $e^2$.
    Moreover, (3) and (4) follow from the fact that the maximum $2$-codegree of $\mathcal{H}$ with respect to $\mathcal{F}$ and the maximum common $2$-codegree of $\mathcal{F}$ are both less than $\Delta_2(\mathcal{F})$.

    Given the two claims above, it follows by \cref{thm:conflict-free} that there exists an $\mathcal{F}$-free matching $\mathcal{M}$ in $\mathcal{H}$ of size at least $$(1-D^{-\alpha})\frac{|E(\mathcal{H})|}{D}=(1-n^{t-w})\frac{\binom{n}{n-w,w_1,\dots, w_{q-1}}}{\binom{n-t}{w-t}f(\overline{w},w-t)},$$ completing the proof of the lemma.
\end{proof}

\begin{proof}[Proof of \cref{thm:ccc}]
    \cref{thm:ccc} is a direct consequence of Lemmas \ref{lem:matching-ccc} and \ref{lem:large-matching-via-Conflict-free}.
\end{proof}

\section{Concluding remarks}\label{sec:con}
%In this paper, for fixed $q$, wight $w$ and odd $d$, sufficiently large $n$, we convert the problem of constructing a large CWC and CCC to the problem of finding a large matching in an auxiliary hypergraph. Applying Kahn's theorem and conflict-free hypergraph matching, we construct a CWC and CCC asymptotically achieving the Johnson-type upper bound.
% we construct an asymptotically optimal constant weight code with large number by Kahn's theorem. By conflict-free hypergraph matching, we get an asymptotically optimal constant composition code. They asymptotically attain the Johnson bound of constant weight codes and constant composition codes, respectively.

\noindent Let us conclude this paper with several interesting problems.

\begin{question}
 Theorems \ref{thm:cwc} and \ref{thm:ccc} showed that the Johnson-type bounds for $A_q(n,d,w)$ and $A_q(n,d,\overline{w})$ are asymptotically tight for all fixed odd distances $d$. Do similar results hold for even $d$?
\end{question}

\begin{question}
 The lower bounds given by Theorems \ref{thm:cwc} and \ref{thm:ccc} are non-constructive. Can we provide explicit constructions with similar lower bounds?
\end{question}

\begin{question}
    Keevash \cite{Kee14} showed that one can precisely attain the Johnson-type bound for $A_2(n,d,w)$ when $n$ is sufficiently large and satisfies certain necessary divisibility conditions.
    %\cref{thm:cwc} showed that for $q\ge 3$, one can asymptotically attain the Johnson-type bound for $A_q(n,d,w)$ when $d$ is odd.
    For all $q\ge 3$ and odd $d$, can we also attain the Johnson-type bound for $A_q(n,d,w)$ with equality for all large $n$ satisfying the necessary divisibility conditions? More on this question can be found in \cite{Etz22}.
\end{question}

\section*{Acknowledgements}

\noindent This project is supported by the National Key Research and Development Program of China under Grant No. 2021YFA1001000, the National Natural Science Foundation of China under Grant Nos. 12101364 and 12231014, and the Natural Science Foundation of Shandong Province under Grant No. ZR2021QA005.

%\bibliographystyle{IEEEtranS}
% \nocite{*}
\bibliographystyle{plain}
\normalem
\bibliography{main}

\end{document}